\documentclass[9pt,final,twocolumn]{IEEEtran}
\long\def\comment#1{}

\usepackage{textcomp}
\usepackage{cite}
\usepackage{graphicx}
\usepackage{subfigure}
\usepackage{psfrag}
\usepackage{url}
\usepackage{stfloats}
\usepackage{amsmath}
\usepackage{amssymb,amsfonts}
\usepackage{amsthm}
\usepackage{float}
\usepackage[usenames]{color}
\usepackage{algorithm,algorithmic}

\newtheorem{assumption}{Assumption}
\newtheorem{remark}{Remark}

\newtheorem{lemma}{Lemma}
\newtheorem{theorem}{Theorem}

\begin{document}

\setlength{\arraycolsep}{0.3em}

\title{Distributed Proximal Algorithms for Multi-Agent Optimization with Coupled Inequality Constraints
\thanks{}}

\author{Xiuxian Li, {\em Member, IEEE}, Gang Feng, {\em Fellow, IEEE}, and Lihua Xie, {\em Fellow, IEEE}
\thanks{This work was supported by grants from the Research Grants Council of Hong Kong (No. CityU-11207817), and the Ministry of Education of Singapore under Grant MoE Tier 1 RG72/19.}
\thanks{X. Li was with Department of Biomedical Engineering, City University of Hong Kong, Kowloon, Hong Kong, P. R. China, where this work was carried out. He is currently with School of Electrical and Electronic Engineering, Nanyang Technological University, 50 Nanyang Avenue, Singapore 639798 (e-mail: xxli@ieee.org).}
\thanks{G. Feng are with Department of Biomedical Engineering, City University of Hong Kong, Kowloon, Hong Kong, P. R. China (e-mail: megfeng@cityu.edu.hk).}
\thanks{L. Xie is with School of Electrical and Electronic Engineering, Nanyang Technological University, 50 Nanyang Avenue, Singapore 639798 (e-mail: elhxie@ntu.edu.sg).}
}

\maketitle

\setcounter{equation}{0}
\setcounter{figure}{0}
\setcounter{table}{0}
%%%%%%%%%%%%%%%%%%%%%%%%%%%%%%%%%%%%%%%%%%%%%%%%%%%%%%%%%%%%%%%%%%%%%%%%%%%%%%%%%%%%%%%%%

\begin{abstract}
This paper aims to address distributed optimization problems over directed and time-varying networks, where the global objective function consists of a sum of locally accessible convex objective functions subject to a feasible set constraint and coupled inequality constraints whose information is only partially accessible to each agent. For this problem, a distributed proximal-based algorithm, called distributed proximal primal-dual (DPPD) algorithm, is proposed based on the celebrated centralized proximal point algorithm. It is shown that the proposed algorithm can lead to the global optimal solution with a general stepsize, which is diminishing and non-summable, but not necessarily square-summable, and the saddle-point running evaluation error vanishes proportionally to $O(1/\sqrt{k})$, where $k>0$ is the iteration number. Finally, a simulation example is presented to corroborate the effectiveness of the proposed algorithm.
\end{abstract}


\section{Introduction}\label{s1}

Distributed optimization has become an active research topic in recent years, mostly inspired by its numerous applications in machine learning, sensor networks, energy systems, and resource allocation \cite{bullo2009distributed}. Until now, a large number of algorithms have been developed, which, in general, can be classified into two categories: consensus-based algorithms and dual-decomposition-based algorithms. Generally speaking, a consensus-based algorithm is to directly integrate consensus theory into an optimization algorithm which only involves primal decision variables, and the distributed algorithms along this line subsume distributed subgradient \cite{nedic2009distributed}, distributed primal-dual subgradient algorithms \cite{zhu2012distributed}, distributed quasi-monotone subgradient algorithm \cite{liang2019distributed}, asynchronous distributed gradient \cite{xu2018convergence}, Newton-Raphson consensus \cite{zanella2011newton}, dual averaging \cite{duchi2012dual}, diffusion adaptation strategy \cite{chen2012diffusion}, fast distributed gradient \cite{jakovetic2014fast}, and stochastic mirror descent \cite{yuan2018optimal}. On the other hand, the dual-decomposition-based algorithms aim at handling the alignment of all local decision variables by equality constraints, through introducing corresponding dual variables, and typical algorithms include augmented Lagrangian method \cite{jakovetic2015linear}, distributed dual proximal gradient \cite{notarnicola2017asynchronous}, EXTRA \cite{shi2015extra}, and distributed forward-backward Bregman splitting \cite{xu2018bregman}.

It is well known that the proximal point algorithm (PPA) is one of important approaches capable of increasing the convergence rate to as fast as $O(1/k)$ for general convex functions \cite{guler1991convergence}, where $k$ is the iteration number. This thus inspires researchers to generalize PPA to distributed optimization problems. Proximal minimization is to add a penalty quadratic term to the original objective function, and can be viewed as an alternative to subgradient approaches. As pointed out in \cite{margellos2018distributed}, this is intriguing itself, because it establishes the relationship between proximal algorithms and gradient methods in the multi-agent scenario, which has been well developed for the case of a single agent \cite{bertsekas1989parallel}. Furthermore, in contrast to incremental algorithms, proximal minimization usually results in numerically more stable algorithms than their gradient-based counterparts \cite{bertsekas2011incremental}. Along this line, proximal minimization was incorporated into the ADMM algorithm to update local decision variables in \cite{aybat2018distributed}, where distributed composite convex optimization is studied under the assumption that the communication graph among agents is fixed and undirected. It was shown that the two proposed algorithms, i.e., deterministic and stochastic distributed proximal gradient algorithms, converge with rates $O(1/k)$ and $O(1/\sqrt{k})$, respectively. Proximal minimization was also employed in \cite{margellos2018distributed} for distributed optimization with feasible constraint sets in uncertain networks, where the convergence to some minimizer was established, but without providing results on the convergence speed. Distributed proximal gradient algorithms were also developed for tackling composite objective functions in \cite{hong2017stochastic,shi2015proximal}.

It should be noted that the aforementioned literature deals with distributed optimization problems under balanced communication graphs. For unbalanced interaction graphs, several approaches have been brought forward in the literature, including push-sum method \cite{nedic2015distributed,nedic2016stochastic,xi2017dextra}, weight balancing method \cite{makhdoumi2015graph}, ``surplus''-based method \cite{xi2017distributed}, row-stochastic matrix method \cite{xi2018linear,li2018distributed}, and epigraph method \cite{xie2018distributed}. Note that the studied problem in \cite{nedic2015distributed} (resp. \cite{nedic2016stochastic}) is for convex functions (resp. strongly convex functions with Lipschitz gradients) with rate $O(\ln k/\sqrt{k})$ (resp. $O((\ln k)/k)$) under time-varying graphs and without constraints, and \cite{xi2017dextra} addressed the case of restricted strongly convex functions and achieves a linear convergence rate under fixed graphs and without constraints. Among these methods, the essence of push-sum and weight balancing strategies is to introduce a scalar variable for each agent to counteract the imbalance of graphs; the ``surplus''-based idea is to introduce an additional column-stochastic matrix and a surplus variable for each agent to conquer the imbalance; the row-stochastic matrix approach generates a network-size variable for each agent to account for the imbalance; and the epigraph method aims to transform the original optimization problem into the epigraph form by introducing a network-size variable for each agent. However, all these methods have their shortcomings. To be specific, each agent needs to know its out-degree for the push-sum and weight-balancing methods; some sort of global information on eigenvalues of the communication graph is required in the ``surplus''-based method; and a network-size variable, which has extremely high dimension for large-scale networks, is stored, transmitted, and updated by each agent when using those methods in \cite{xie2018distributed,xi2018linear,li2018distributed}.

Inspired by the above observations, this paper investigates distributed convex optimization problems with a feasible set constraint and coupled inequality constraints under time-varying interaction graphs, where all involved functions are only assumed to be convex. For this problem, a distributed proximal-based algorithm is proposed, and the convergence analysis of the algorithm is also provided. The contributions of this paper can be summarized in the following aspects.

\begin{enumerate}
  \item A distributed algorithm, named distributed proximal primal-dual (DPPD) algorithm, is developed for the concerned problem and proved to be convergent to the optimizer set.
  \item The convergence rate is analyzed in the sense of the saddle-point running evaluation error, which is shown to decrease at the rate of $O(1/\sqrt{k})$, where $k>0$ is the iteration number.
  \item It is also worthwhile to note that the stepsize here is not necessarily to be square-summable as usually required in the literature (exceptions include \cite{liu2017convergence,qiu2018necessary,wang2018distributed}) by using the idea in \cite{liu2017convergence}, where a different algorithm is considered without inequality constraints.
\end{enumerate}

The remainder of this paper is organized as follows. Preliminaries as well as problem statement are provided in Section II. Section III provides the main results of this paper, and a simulation example is presented for validating the proposed algorithm in Section IV. Finally, the conclusion is drawn in Section V.

{\em Notations:} Denote by $\mathbb{R}_+^n$ the set of $n$-dimensional vectors with nonnegative components, and $[N]:=\{1,2,,\ldots,N\}$ the index set for an integer $N>0$. Let $col(z_1,\ldots,z_k)$ be the stacked column vector of $z_i\in\mathbb{R}^n,i\in [k]$. Let $\|\cdot\|$, $\|\cdot\|_1$, $x^\top$ and $\langle x,y\rangle$ stand for the standard Euclidean norm, $\ell_1$-norm, the transpose of a vector $x$ and the standard inner product of $x,y\in\mathbb{R}^n$, respectively. Denote by $P_X(z)$ the projection of a point $z\in\mathbb{R}^n$ onto the set $X\subset\mathbb{R}^n$, i.e., $P_X(z):=\mathop{\arg\min}_{x\in X}\|z-x\|$, and let $[z]_+$ be the component-wise projection of a vector $z\in\mathbb{R}^n$ onto $\mathbb{R}^n_+$. In addition, let $I$ be the identity matrix of compatible dimension, and $\otimes$ be the Kronecker product. And define $\|y\|_X$ to be the distance from a point $y$ to the set $X$, i.e., $\|y\|_X:=\inf_{x\in X}\|y-x\|$. Let $\lfloor c\rfloor$ be the largest integer less than or equal to a real number $c$.

\section{Preliminaries and Problem Statement}\label{s2}

\subsection{Convex Optimization}\label{s2.1}

Given a function $h:\mathbb{R}^n\to\mathbb{R}$, the {\em proximal operator} $\textbf{prox}_h:\mathbb{R}^n\to\mathbb{R}^n$ of $h$ is defined as
\begin{align}
\textbf{prox}_h(v)=\mathop{\arg\min}_{x\in\mathbb{R}^n}\big(h(x)+\frac{1}{2}\|x-v\|^2\big),      \label{c2}
\end{align}
which is assumed to be efficiently computable whenever employed throughout this paper. For example, $\textbf{prox}_h(v)=(I+P)^{-1}(v-q)$ for the quadratic function $h(x)=x^\top Px/2+q^\top x+r$ with $P$ being positive semi-definite and $\textbf{prox}_h(v)_i=(v_i+\sqrt{v_i^2+4})/2$ for logarithmic barrier $h(x)=-\sum_{i=1}^n\log x_i$, where $\textbf{prox}_h(v)_i$ means the $i$-th component of $\textbf{prox}_h(v)$.

We call a function $L:X\times\Lambda\to\mathbb{R}$ {\em convex-concave} if $L(\cdot,\lambda):X\to\mathbb{R}$ is convex for each $\lambda\in\Lambda$ and meanwhile $L(x,\cdot):\Lambda\to\mathbb{R}$ is concave for each $x\in X$, where $X\subset\mathbb{R}^n,\Lambda\subset\mathbb{R}^m$. A {\em saddle point} of the convex-concave function $L$ over $X\times\Lambda$ is defined to be a pair $(x^*,\lambda^*)$ such that
\begin{align}
L(x^*,\lambda)\leq L(x^*,\lambda^*)\leq L(x,\lambda^*),~~~~\forall x\in X,\lambda\in\Lambda.         \label{c3}
\end{align}

To seek the saddle point of the convex-concave function $L$ is usually called the {\em saddle-point problem}, {\em minimax problem} or {\em min-max problem}. Given the function $L(x,\lambda)$ for $x\in\mathbb{R}^n$ and $\lambda\in\mathbb{R}^m$, let $\partial_x L$ and $\partial_\lambda L$ be the subgradients with respect to $x$ and $\lambda$, respectively. For more details, please refer to \cite{bertsekas2003convex}.

\subsection{Problem Statement}\label{s2.2}

In this paper, we consider a network consisting of $N$ agents, which cooperatively solve the following minimization problem, i.e.,
\begin{align}
\min_{x\in X_0}~~f(x):=\sum_{i=1}^N f_i(x),~~~~\text{s.t.}~~g(x):=\sum_{i=1}^N g_i(x)\leq 0,                 \label{1}
\end{align}
where $f,g:\mathbb{R}^{n}\to\mathbb{R}$ are the global objective and constraint functions, respectively, $f_i:\mathbb{R}^{n}\to\mathbb{R}$ is the local objective function that is only accessible to agent $i$, and $x\in X_0\subset\mathbb{R}^n$ is the global decision variable. Also, $g_i:\mathbb{R}^n\to\mathbb{R}^{m}$ is only accessible to agent $i$, meaning that each agent has only access to partial information of the global inequality constraints. Note that all inequalities are understood componentwise throughout this paper.

\begin{remark}\label{r01}
Note that, akin to the separability of $f$ usually considered in distributed optimization, $g$ is also separable here in (\ref{1}). Nonetheless, the algorithm in this paper can be adapted to the case where the agents' decision variables are not aligned, i.e., $\min_{x\in X_0} \sum_{i=1}^N f_i(x_i)$, s.t. $\sum_{i=1}^N g_i(x_i)\leq 0$, where $x=col(x_1,\ldots,x_N)$. Moreover, coupled inequality constraints are encountered in a wide range of applications in such as power systems and plug-in electric vehicles charging problems, to name a few, and have been investigated intensively in recent years \cite{li2019distributed,chang2014distributed,falsone2017dual,notarnicola2017duality,mateos2017distributed}.
\end{remark}

In order to model the communications among agents, a digraph is introduced as $\mathcal{G}_k=(\mathcal{V},\mathcal{E}_k)$ at time instant $k$, with $\mathcal{V}=\{1,\ldots,N\}$ and $\mathcal{E}_k\subset\mathcal{V}\times\mathcal{V}$ being the node and edge sets at time step $k$, respectively. Let $(j,i)$ denote an edge in $\mathcal{E}_k$, meaning that node $i$ can receive information from node $j$ at time $k$, and in this case, we call $j$ (resp. $i$) an in-neighbor (resp. out-neighbor) of $i$ (resp. $j$). A graph is called strongly connected if any node can be connected to any other node by a directed path, where a directed path means a sequence of directed adjacent edges. Define the adjacency matrix $A_k=(a_{ij,k})\in\mathbb{R}^{N\times N}$ at time $k$ with $a_{ij,k}>0$ if $(j,i)\in\mathcal{E}_k$, and $a_{ij,k}=0$ otherwise.

To proceed further, several assumptions on the distributed optimization problem (\ref{1}) are imposed as follows.

\begin{assumption}[Communication and Connectivity]\label{a1}
For all $k\geq 0$,
\begin{enumerate}
  \item $a_{ii,k}\geq a$ for all $i\in[N]$, and $a_{ij,k}\geq a$ if $a_{ij,k}>0$, where $a$ is a constant in $(0,1)$;
  \item The matrix $A_k$ is double-stochastic, i.e., $\sum_{j=1}^N a_{ij,k}=1$ for all $i\in [N]$, and furthermore, $\sum_{i=1}^N a_{ij,t}=1$ for all $j\in[N]$;
  \item A constant $Q>0$ exists such that the union graph $(\mathcal{V},\cup_{l=0,\ldots,Q-1}\mathcal{E}_{k+l})$ is strongly connected for all $k\geq 0$.
\end{enumerate}
\end{assumption}

\begin{assumption}[Convexity and Compactness]\label{a2}~~~~~~~~~~~
\begin{enumerate}
  \item The functions $f_{i}$ and $g_{i}$ are convex for all $i\in[N]$.
  \item The set $X_0$ is closed, convex and compact.
\end{enumerate}
\end{assumption}

From Assumption \ref{a2}, one can apparently see that all functions are not required to be differentiable. In the meantime, in light of the compactness of $X_0$, which is of interest to many practical problems since variables in reality are always bounded, there must exist constants $D,E,S>0$ such that for all $x\in X_0$ and all $i\in[N]$
\begin{align}
&\|x\|\leq D,                                               \label{4}\\
&\|f_{i}(x)\|\leq E,~~~\|g_i(x)\|\leq E,                        \label{5}\\
&\|\partial f_{i}(x)\|\leq S,~~\|\partial g_{i}(x)\|\leq S.         \label{6}
\end{align}

\begin{assumption}[Slater Condition]\label{a3}
Consider problem (\ref{1}). There exists a point $\check{x}\in relint(X_0)$, where $relint(\cdot)$ means the relative interior of a set, such that $\sum_{i=1}^N g_i(\check{x})<0$.
\end{assumption}

A vector satisfying Slater condition is often called {\em slater vector}.

\section{Main Results}\label{s4}

\subsection{Distributed Algorithm and Convergence Analysis}\label{s4.1}

For problem (\ref{1}), the Lagrangian function is in the following form
\begin{align}
L(x,\mu)&=\sum_{i=1}^N f_i(x)+\mu^\top\sum_{i=1}^N g_i(x)=\sum_{i=1}^N L_i(x,\mu),          \label{4s1}
\end{align}
where $x\in\mathbb{R}^n$ is the global decision variable, $\mu\in\mathbb{R}^{m}$ is the dual variable or Lagrange multiplier associated with inequalities (\ref{1}), and
\begin{align}
L_i(x,\mu):=f_i(x)+\mu^\top g_i(x)        \label{zzz}
\end{align}
is the local Lagrangian function for $i\in[N]$.

By virtue of the proximal method, an algorithm, called distributed proximal primal-dual (DPPD) algorithm, is proposed, as given in Algorithm 1, where $\alpha_k$ is a nonincreasing stepsize, satisfying
\begin{align}
\alpha_k>0,~~\lim_{k\to\infty}\alpha_k=0,~~\sum_{k=1}^\infty \alpha_k=+\infty.           \label{ss3}
\end{align}

Furthermore, in (\ref{ag3-3}), $U$ is a bounded subset of $\mathbb{R}_+^{m}$ and supposed to contain the optimal dual set. At this stage, the set $U$ is just employed as a priori knowledge, whose computation is delayed to the next subsection.

It is easy to verify that (\ref{ag3-3}) is equivalent to
\begin{align}
\mu_{i,k+1}=P_U [\hat{\mu}_{i,k}+\alpha_k g_i(x_{i,k+1})].           \label{4s00}
\end{align}

\begin{algorithm}
 \caption{Distributed Proximal Primal-Dual (DPPD)}
 \begin{algorithmic}[1]
% \renewcommand{\algorithmicrequire}{\textbf{Require:}}
% \renewcommand{\algorithmicensure}{\textbf{Output:}}
% \REQUIRE
%% \ENSURE  out
%%\\  \textit{Initialization}:
  \STATE \textbf{Initialization:} Stepsize $\alpha_k$ in (\ref{ss3}), and local initial conditions $x_{i,0}\in X_0$ and $\mu_{i,0}\in U$ for all $i\in[N]$.
  \STATE \textbf{Iterations:} Step $k\geq 0$: update for each $i\in[N]$:
\begin{align}
&\hat{x}_{i,k}=\sum_{j=1}^N a_{ij,k}x_{j,k},~~~~\hat{\mu}_{i,k}=\sum_{j=1}^N a_{ij,k}\mu_{j,k},                                \label{ag3-1}\\
&x_{i,k+1}=\mathop{\arg\min}_{x_i\in X_0}\Big(L_i(x_i,\hat{\mu}_{i,k})+\frac{1}{2\alpha_k}\|x_i-\hat{x}_{i,k}\|^2\Big),              \label{ag3-2}\\
&\hspace{-0.2cm}\mu_{i,k+1}=\mathop{\arg\max}_{\mu_i\in U}\Big(L_i(x_{i,k+1},\mu_i)-\frac{1}{2\alpha_k}\|\mu_i-\hat{\mu}_{i,k}\|^2\Big).              \label{ag3-3}
\end{align}
 \end{algorithmic}
\end{algorithm}

\begin{remark}\label{rrr}
Note that a distributed proximal-based algorithm has been proposed in \cite{margellos2018distributed} for handling distributed optimization problems with feasible set constraints, but no coupled inequality constraints are addressed and meanwhile it is considered without the analysis on the convergence rate.
\end{remark}

To proceed, let us first present some preliminary results on consensus of local variables $x_{i,k}$'s and $\mu_{i,k}$'s for $i\in[N]$.

\begin{lemma}\label{l2}
If Assumptions \ref{a1}-\ref{a2} hold with $\alpha_k$ given in (\ref{ss3}), then there hold
\begin{align}
\|x_{i,k}-\bar{x}_k\|&=O(\alpha_{\lfloor\frac{k}{2}\rfloor}),                \label{ass15}\\
\|\mu_{i,k}-\bar{\mu}_{k}\|&=O(\alpha_{\lfloor\frac{k}{2}\rfloor}),~~~\forall i,j\in[N]         \label{ass0}
\end{align}
where $\bar{x}_k:=\frac{1}{N}\sum_{i=1}^N x_{i,k}$ and $\bar{\mu}_k:=\frac{1}{N}\sum_{i=1}^N \mu_{i,k}$.
\end{lemma}
\begin{proof}
The proof can be found in the Appendix A.
\end{proof}

To facilitate the subsequent analysis, we denote by $X^*$ and $U^*$ the optimal primal and dual variable sets, respectively, and let $f^*$ be the optimal value of the cost function $f$ corresponding to the optimal sets $X^*$ and $U^*$. Then, it is easy to observe that $L(x^*,\mu^*)=f^*$ for any point $x^*\in X^*,\mu^*\in U^*$ since $(\mu^*)^\top\sum_{i=1}^N g_i(x^*)=0$ by first-order optimality conditions. Note that each optimal variable in $X^*$ obviously satisfies the constraints in (\ref{1}). Moreover, similar to \cite{mateos2017distributed}, the {\em running evaluation error} for measuring the convergence rate to the optimal value is defined as
\begin{align}
\Big|\frac{\sum_{l=1}^k L(\bar{x}_{l+1},\bar{\mu}_{l+1})}{k}-f^*\Big|,        \label{ee}
\end{align}
where $\bar{x}_k$ and $\bar{\mu}_k$ are defined in Lemma \ref{l2}.

We are now in a position to present the main results.

\begin{theorem}\label{t3}
If Assumptions \ref{a1}-\ref{a3} are satisfied with $\alpha_k$ given in (\ref{ss3}), then under Algorithm 1, all $x_{i,k}$'s will reach a common point in $X^*$ and meanwhile $\mu_{i,k}$'s will reach a common point in $U^*$.

Moreover, the following holds for the running evaluation error:
\begin{align}
\Big|\frac{\sum_{l=1}^k L(\bar{x}_{l+1},\bar{\mu}_{l+1})}{k}-f^*\Big|=O(1/\sqrt{k}),            \label{4s2}
\end{align}
when setting $\alpha_k=1/\sqrt{k}$ for $k\geq 1$.
\end{theorem}

\begin{proof}
The proof can be found in the Appendix B.
\end{proof}

\begin{remark}\label{r3}
The proposed algorithm, described in Algorithm 1, has two main features: 1) it employs a general stepsize, i.e., satisfying (\ref{ss3}), not necessarily square-summable; and 2) it converges with rate $O(1/\sqrt{k})$ in the sense of running evaluation error. Note that the stepsize's square-summability is not required as well in \cite{liu2017convergence}, but a different algorithm, i.e., projected subgradient algorithm, was studied in \cite{liu2017convergence} without inequality constraints. In contrast, a proximal-based algorithm is developed here with coupled inequality constraints. Moreover, our results are advantageous in contrast with \cite{chang2014distributed} where local objective functions are assumed to be continuously differentiable, the global objective function is supposed to have Lipschitz continuous gradients, and no convergence rate is reported. In \cite{falsone2017dual}, the convergence is given in the ergodic sense, and no convergence speed is given, although non-identical feasible sets are discussed. Fixed and undirected graphs are considered in \cite{notarnicola2017duality}, and no convergence speed is given either. It should be also noted that the algorithms in \cite{chang2014distributed,falsone2017dual,notarnicola2017duality} are intrinsically different from DPPD in this paper. Additionally, the proximal minimizations in Algorithm 1 do not need the computation of (sub)gradients, which is often the case for most of other algorithms, and thus Algorithm 1 is more computationally tractable.
\end{remark}

\begin{remark}\label{r1v1}
It should be noted that the result can be generalized to handle the case when each agent has its own local decision variable in problem (\ref{1}). Moreover, it is worthwhile to point out that Algorithm 1 can be easily extended to solve two saddle-point problems. Specifically, the first is $\min_{x\in X}\max_{y\in Y} H(x,y):=\sum_{i=1}^N H_i(x,y_i)$, where $H$ is the global convex-concave objective function, $H_i(x,y_i)$ is the local convex-concave objective function defined over $X\times Y_i$ that is only accessible to agent $i$, $x,y_i$ are the global and local decision variables, respectively, $y:=col(y_1,\ldots,y_N)$, $Y:=Y_1\times\cdots\times Y_N$, and $X\subset\mathbb{R}^n,Y_i\subset\mathbb{R}^{m_i}$ are some sets. Note that $X$ is a common information for all agents and $Y_i$ is locally known to agent $i$. The second is $\min_{x\in \mathcal{X}}\max_{y\in \mathcal{Y}} \Xi(x,y):=\sum_{i=1}^N \Xi_i(x,y)$, where $\mathcal{X}\subset\mathbb{R}^n$ and $\mathcal{Y}\subset\mathbb{R}^{m}$ are nonempty sets and commonly known by all agents, $x,y$ are the global decision variables, $\Xi$ is the global convex-concave objective function, and $\Xi_i(x,y)$ is the local convex-concave objective function defined over $\mathcal{X}\times \mathcal{Y}$ that is only accessible to agent $i$. The details are omitted due to limited space.
\end{remark}

\subsection{Bound on Optimal Dual Set}\label{s4.2}

In the last subsection, the bounded set $U$ has been exploited for studying dual variables $\mu_{i,k}$, which has also been employed in \cite{nedic2009approximate,chang2014distributed,mateos2017distributed}. In this subsection a distributed strategy is developed to obtain the set $U$.

Let us first establish a bound on the dual variable $\mu$ corresponding to inequality constraints $\sum_{i=1}^N g_i(x)\leq 0$. To this end, let $q(\mu)$ denote the dual function defined as
\begin{align}
q(\mu)&:=\inf_{x\in X_0} \sum_{i=1}^N f_i(x)+\mu^\top\sum_{i=1}^N g_i(x)          \nonumber\\
&\geq \sum_{i=1}^N \big(\inf_{x\in X_0} f_i(x)+\mu^\top g_i(x)\big)=:\sum_{i=1}^N q_i(\mu).             \label{4s9}
\end{align}
Now, invoking Lemma 1 in \cite{nedic2009approximate}, it can be asserted that
\begin{align}
\max_{\mu^*\in U^*}\|\mu^*\|&\leq \frac{f(\check{x})-q(\check{\mu})}{\gamma}            \nonumber\\
&\leq \frac{N[\max_{i\in[N]}f_i(\check{x})-\min_{i\in[N]}q_i(\check{\mu})]}{\gamma}=:U_c,          \label{4s10}
\end{align}
where $\check{x}$ is a slater vector and $\check{\mu}\in\mathbb{R}_+^m$ is any vector. Moreover,
\begin{align}
\gamma:=\min_{l\in[m]}\Big\{-\sum_{i=1}^N g_{il}(\check{x})\Big\},        \label{4s11}
\end{align}
where $g_{il}$ is the $l$-th component of $g_i$. As a consequence, the optimal set $U^*$ is contained in $\Upsilon$, defined as
\begin{align}
\Upsilon:=\Big\{\mu\in\mathbb{R}_+^m:\|\mu\|\leq U_c\Big\}.        \label{4s12}
\end{align}

In what follows, a distributed method, inspired by \cite{mateos2017distributed}, is proposed for each agent to obtain the set $U$ employed in the last subsection such that $\Upsilon\subset U$. It includes three steps.

{\em Step 1:} Each agent finds the slater vector $\check{x}$ by solving
\begin{align}
\min_{x\in X_0}~~\sum_{i=1}^N g_i(x),                 \label{4s13}
\end{align}
using the distributed algorithm
\begin{align}
x_{i,k+1}=\mathop{\arg\min}_{x_i\in X_0}\Big(g_i(x_i)+\frac{1}{2\alpha_k}\|x_i-\hat{x}_{i,k}\|^2\Big),                 \label{4s14}
\end{align}
with any initial condition $x_{i,0}\in X_0$. Note that (\ref{4s14}) can be considered as a special case of Algorithm 1 when $g_i\equiv 0$ for all $i\in[N]$. Therefore, Theorem \ref{t3} holds for (\ref{4s14}), and the final convergent point, say $\check{x}$, must satisfy $\sum_{i=1}^N g_i(\check{x})<0$ due to Slater condition. In addition, each agent can independently compute $q_i(\mu)=\inf_{x\in X_0} f_i(x)+\mu^\top g_i(x)$, and the common point $\check{\mu}\in\mathbb{R}_+^m$ in (\ref{4s10}) can be selected as any point agreed upon by all agents in advance, such as $\check{\mu}=0_m$.

{\em Step 2:} All agents seek a common lower bound on $\gamma$ in (\ref{4s11}). To this end, each agent needs to ensure $g_{i}$ to be negative through a consensus algorithm, and simultaneously finds a maximum of all $g_i$'s by another finite-time consensus algorithm.

Specifically, setting $z_{i,0}=g_i(\check{x})$, which may be nonnegative, and $s_{i,0}=z_{i,0}$ for $i\in[N]$. Each agent updates its variables $z_{i,k}$ and $s_{i,k}$ for $k\geq 0$ by
\begin{align}
z_{i,k+1}&=\sum_{j=1}^N a_{ij,k}z_{j,k},                 \label{4s15}\\
s_{i,k+1}&=\max \{s_{j,k}:j\in\mathcal{N}_{i,k}^+\}.        \label{4s17}
\end{align}
(\ref{4s15}) can be written in a compact form
\begin{align}
z_{k+1}=(A_k\otimes I_m) z_{k},                 \label{4s16}
\end{align}
where $z_k:=col(z_{1,k},\ldots,z_{N,k})$, and the $(i,j)$-th entry of $A_k$ is $a_{ij,k}$ for $i,j\in[N]$.

Updating equation (\ref{4s17}) will achieve consensus with the ultimate value being $\max_{i\in[N]}s_{i,0}=\max_{i\in[N]}z_{i,0}\in\mathbb{R}^m$ in finite iterations no greater than $(N-1)Q$. By setting $\varsigma=(N-1)Q$, at time step $\varsigma$, one then has $s_{i,\varsigma}=\max_{i\in[N]}z_{i,0}$ for all $i\in[N]$. If $sign(s_{i,\varsigma})\leq -1_m\otimes \tau$, then stop the iterations on $s_{i,k}$'s, where $\tau$ is any pre-specified constant in $(0,1)$ and $sign(\cdot)$ is the standard signum function, componentwise for a vector; otherwise, each agent $i$ re-initializes $s_{i,0}=z_{i,\varsigma}$ and updates $s_{i,k}$ for $\varsigma$ iterations according to (\ref{4s17}) again. Similarly, it can be obtained that $s_{i,\varsigma}=\max_{i\in[N]}z_{i,\varsigma}$. At this time, we stop updates of $s_{i,k}$'s if $sign(s_{i,\varsigma})\leq -1_m\otimes \tau$, and otherwise each agent $i$ re-initializes $s_{i,0}=z_{i,2\varsigma}$ and updates $s_{i,k}$ for $\varsigma$ iterations again according to (\ref{4s17}). Repeating this process, it can be asserted that the process will be eventually terminated in a finite time, as shown in the following result, and the final consensus value on $s_{i,k}$'s is denoted as $\check{z}\in\mathbb{R}^m$, satisfying $sign(\check{z})\leq -1_m\otimes \tau$.

\begin{lemma}\label{l3}
For dynamics (\ref{4s15}), there exists a finite constant $k^*>0$ such that $sign(z_{i,k})\leq -1_m\otimes \tau$ for all $i\in[N]$ and all $k\geq k^*$.
\end{lemma}

\begin{proof}
According to Lemma 1 in \cite{nedic2015distributed} without perturbations or errors, it can be concluded that $z_{i,k+1}\to \sum_{j=1}^N z_{j,k}/N$ exponentially. Due to the double-stochastic property of $D_k$, left-multiplying $1_{N}^\top\otimes I_m$ on both sides of (\ref{4s16}) implies that $\sum_{i=1}^N z_{i,k+1}=\sum_{i=1}^N z_{i,k}$ for all $k\geq 0$, thus giving rise to that $\sum_{i=1}^N z_{i,k}=\sum_{i=1}^N z_{i,0}$ for all $k\geq 0$. Combining the above facts, one can obtain that $z_{i,k}\to \sum_{j=1}^N z_{j,0}/N$ exponentially for each $i\in[N]$. By combining $z_{i,k}\to \omega_{i,\infty}\sum_{j=1}^N z_{j,0}/N$ exponentially with $\sum_{j=1}^N z_{j,0}=\sum_{j=1}^N g_j(\check{x})<0$, one can obtain that there must exist a finite time $k^*$ such that $sign(z_{i,k})\leq -1_m\otimes \tau$ for all $i\in[N]$ and $k\geq k^*$.
\end{proof}

Then, each agent can compute the same lower bound on $\gamma$ as follows
\begin{align}
\gamma&=\min_{l\in[m]}\Big\{-\sum_{i=1}^N g_{il}(\check{x})\Big\}\geq \min_{l\in[m]}\{-N\check{z}_l\}=:\underline{\gamma}.                 \label{4s18}
\end{align}

{\em Step 3:} All agents can reach agreement on $\max_{i\in[N]}f_i(\check{x})$ and $\min_{i\in[N]}q_i(\check{\mu})$ in finite iterations by executing the analogous algorithms to (\ref{4s17}).

Combining the above three steps, each agent can ultimately compute the set $U$ as
\begin{align}
U:=\Big\{\mu\in\mathbb{R}_+^m:\|\mu\|\leq U_0\Big\},         \label{4s19}
\end{align}
where $U_0:=\frac{N[\max_{i\in[N]}f_i(\check{x})-\min_{i\in[N]}q_i(\check{\mu})]}{\underline{\gamma}}$.

\begin{remark}\label{rr1}
It is worthwhile to mention that the distributed strategy, devised in Section V-A in \cite{mateos2017distributed}, is only valid for the case when local decision variables $x_i$'s do not need to be aligned, while the problem (\ref{1}) involves a global decision variable. What's more, in \cite{mateos2017distributed} each agent can find a positive constant $k_i^*$ independently such that its own iteration is terminated, and subsequently all agents reach a consensus on $z_a:=\max_{i\in[N]}z_i(k_i^*)$, but it cannot ensure that $\sum_{i=1}^N g_i(z_a)<0$. To address the issues arising from the problem (\ref{1}), the procedure in Steps 1 and 2 has been redesigned here.
\end{remark}

\section{A Simulation Example}\label{s6}

This example is used to demonstrate the validity of Algorithm 1 for problem (\ref{1}). Motivated by applications in wireless networks \cite{mateos2017distributed}, consider the following distributed constrained optimization problem over a network of $N=100$ agents:
\begin{align}
&\min_{x\in X_0}~\sum_{i=1}^N f_i(x)=\sum_{i=1}^N \theta_i x           \nonumber\\
&\text{s.t.}~~~~\sum_{i=1}^N g_i(x)=\sum_{i=1}^N \Big(-d_i\log(1+x)+\frac{b}{N}\Big)\leq 0,                          \label{se2}
\end{align}
where $\theta_i,d_i$ and $b$ are some constant scalars. It should be noticed that problem (\ref{se2}) is also considered in \cite{mateos2017distributed}, but without the alignment of local variables $x_{i,k}$'s. Set $X_0=[0,1]$, $b=5$, $\theta_i=i/N\in[0,1]$, $d_i=i/(N+1)\in[0,1]$ for $i\in[N]$, $\alpha_k=1/\sqrt{k},\forall k\geq 1$, and two cases for $Q=2$ and $Q=50$. Applying Algorithm 1 gives evolutions of local variables $x_{i,k}$'s and the averaged objective value, as shown in Figs. \ref{f1} and \ref{f2}. For both cases, the algorithms are convergent to the optimizer $x^*=e^{0.1}-1\approx 0.1052$, thus supporting the theoretical result. Note that the optimizer $x^*$ satisfies the constraint in (\ref{se2}). Moreover, the convergence for $Q=50$ is slower than the case of $Q=2$, showing that stronger connectivity has better convergence performance. To compare with the algorithm in \cite{mateos2017distributed}, it is set $Q=1$ for a balanced graph with all other parameters unchanged as above, and the simulation result is given in Fig. \ref{f3}, indicating that the C-SP-SG algorithm in \cite{mateos2017distributed} is slower than DPPD algorithm here. It is also noted that the convergence in \cite{mateos2017distributed} is provided in an ergodic sense.

\begin{figure}[H]
\centering
\subfigure[]{\includegraphics[width=2.6in]{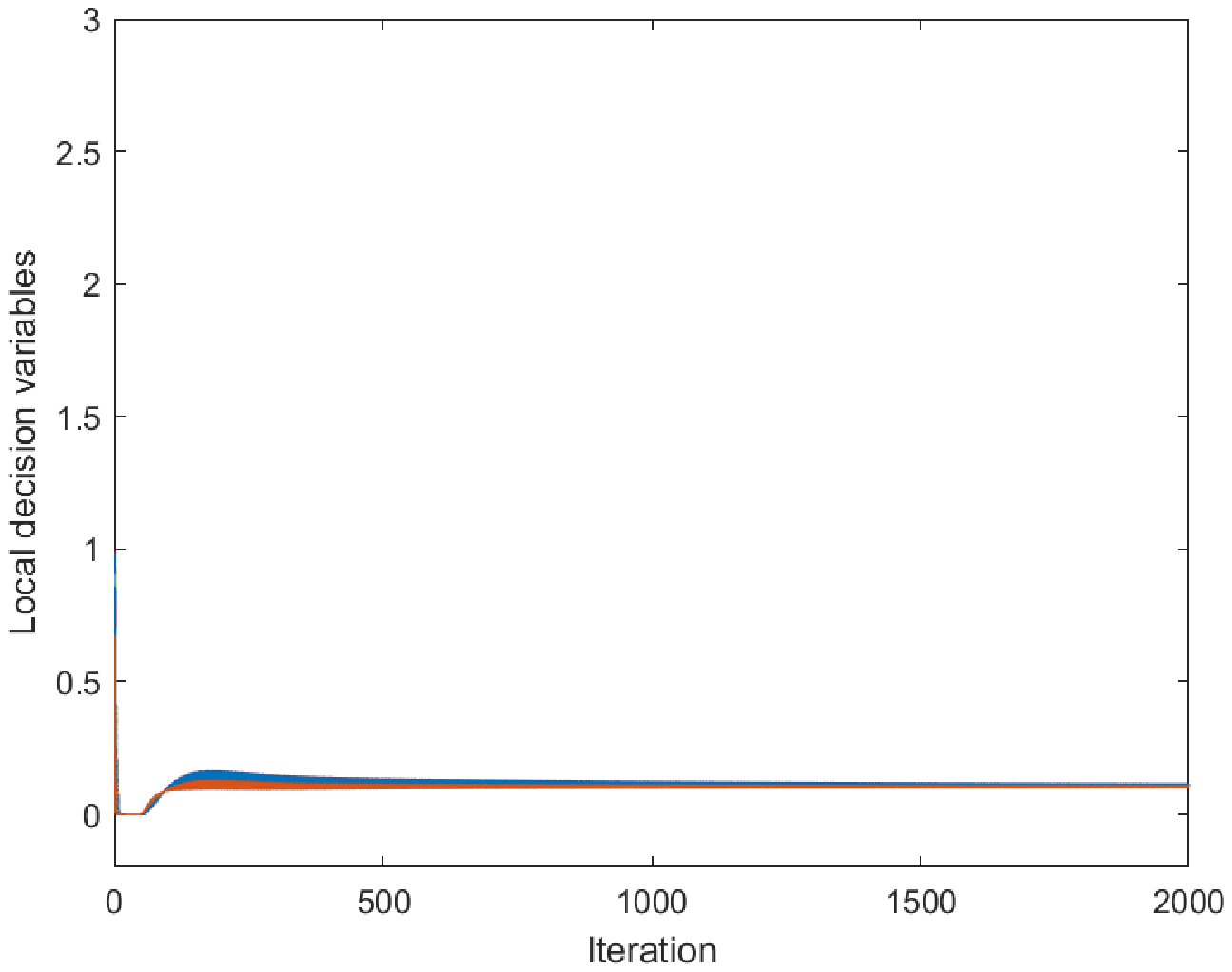}}
\subfigure[]{\includegraphics[width=2.6in]{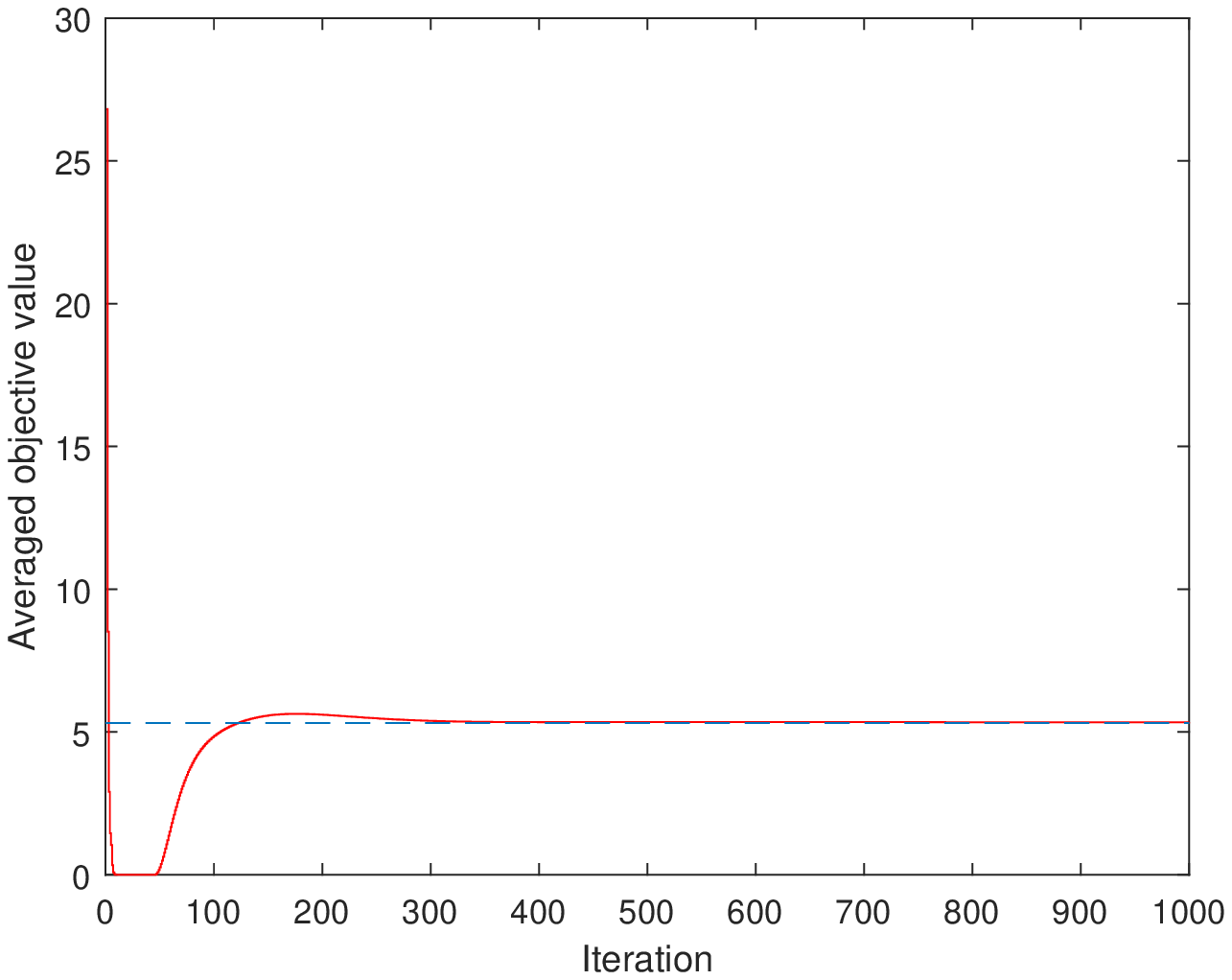}}
\caption{Evolutions of $x_{i,k}$'s in (a) and $\sum_{l=1}^k L(\bar{x}_{l+1},\bar{\mu}_{l+1})/k$ in (b) for $Q=2$, where the dashed line in (b) means the optimal objective value $f^*=50.5(e^{0.1}-1)\approx 5.3111$.}
\label{f1}
\end{figure}

\section{Conclusion}\label{s7}

This paper has investigated distributed convex optimization problems with a feasible set constraint and coupled inequality constraints over time-varying communication graphs. Under mild assumptions, a distributed proximal-based algorithm, called the distributed proximal primal-dual (DPPD) algorithm, was proposed to deal with the considered problem, and the running evaluation error was proved to decrease proportionally to $O(1/\sqrt{k})$, where $k>0$ is the number of iterations. Furthermore, the designed algorithms can be easily generalized to handle more general distributed saddle-point problems. Possible future directions are to address the case of different feasible set constraints, and to attempt to accelerate the convergence speed to that of the centralized proximal point algorithm, i.e., $O(1/k)$.

\begin{figure}[H]
\centering
\subfigure[]{\includegraphics[width=2.6in]{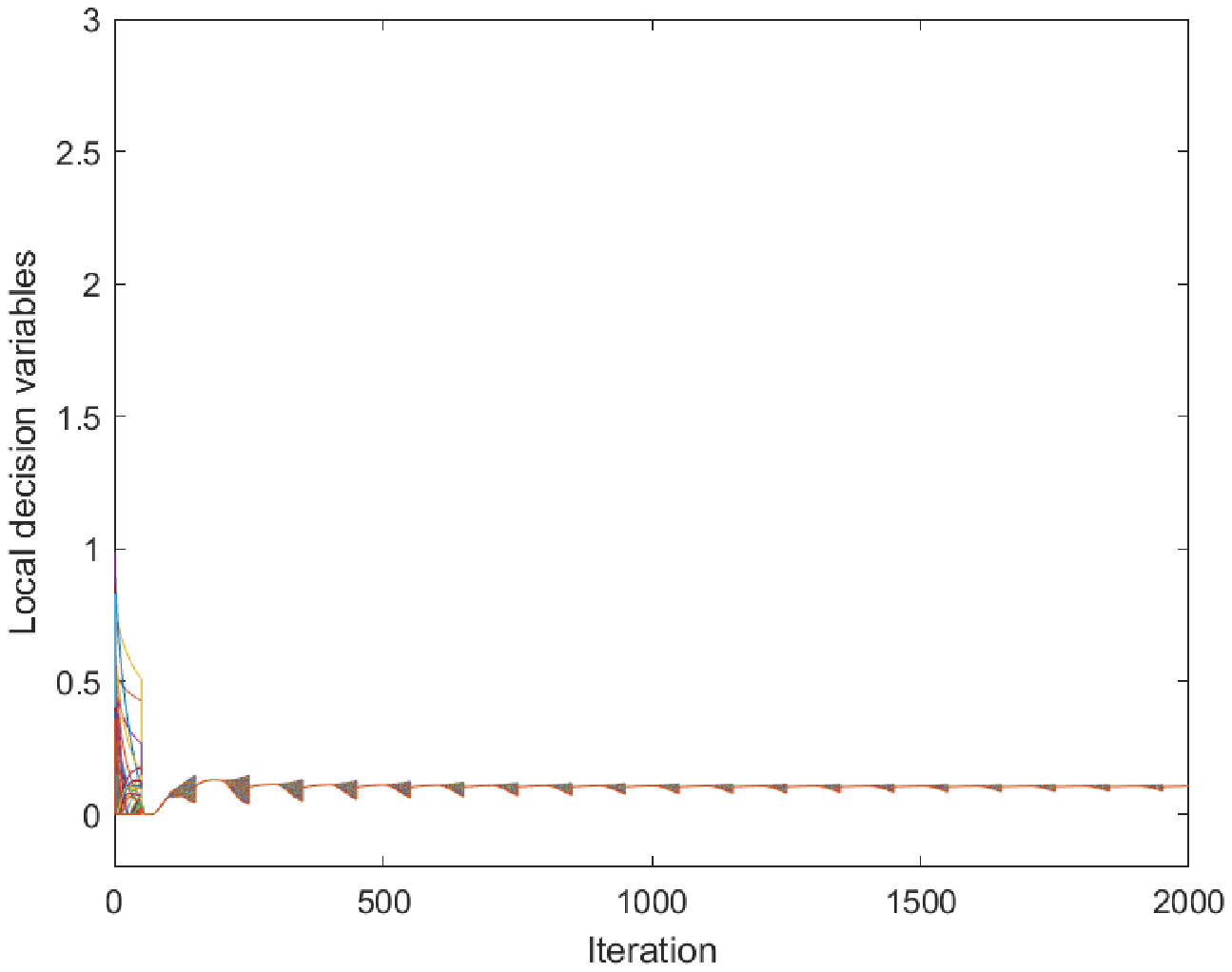}}
\subfigure[]{\includegraphics[width=2.6in]{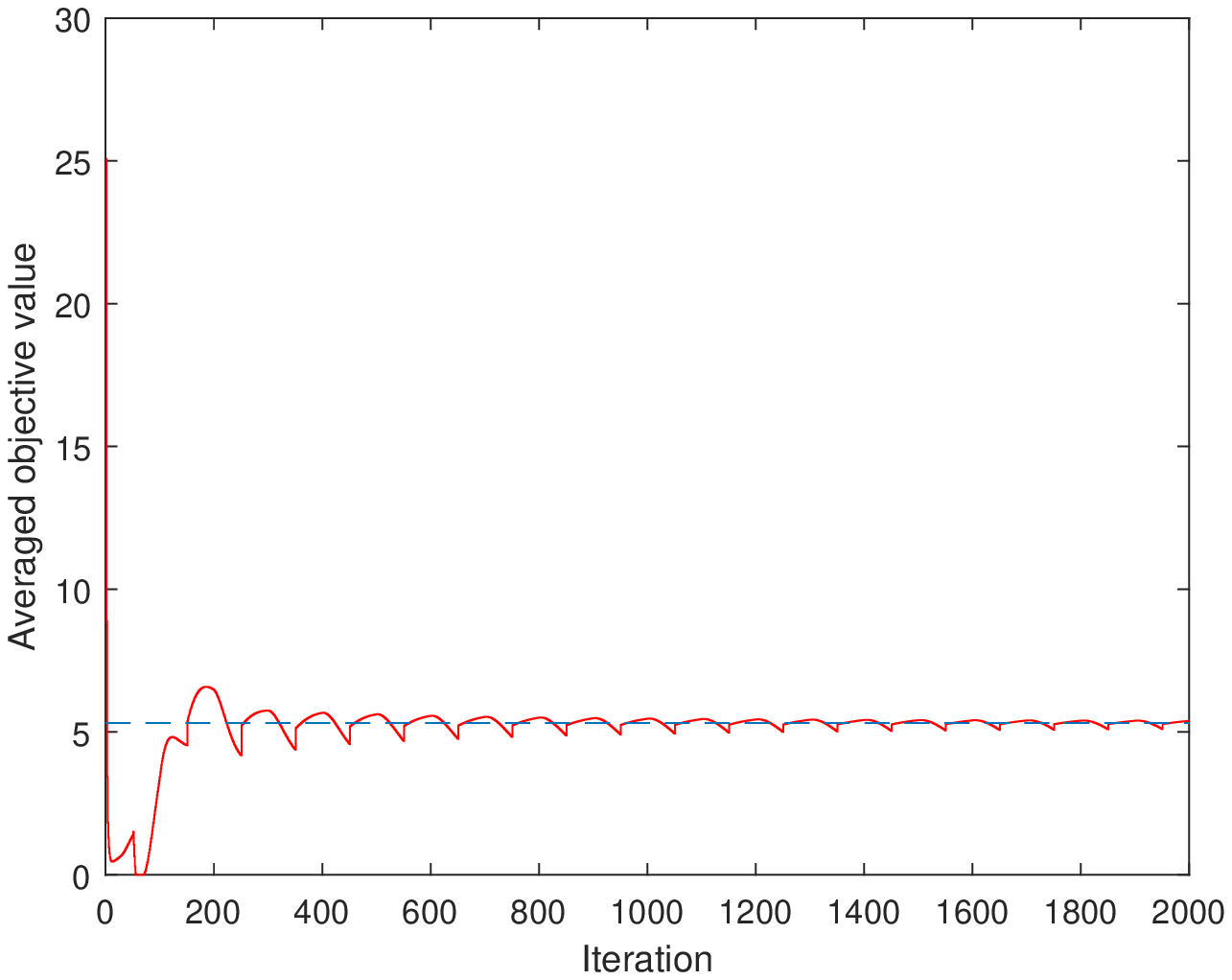}}
\caption{Evolutions of $x_{i,k}$'s in (a) and $\sum_{l=1}^k L(\bar{x}_{l+1},\bar{\mu}_{l+1})/k$ in (b) for $Q=50$, where the dashed line in (b) means the optimal objective value $f^*=50.5(e^{0.1}-1)\approx 5.3111$.}
\label{f2}
\end{figure}

\begin{figure}[H]
\centering
\includegraphics[width=2.6in]{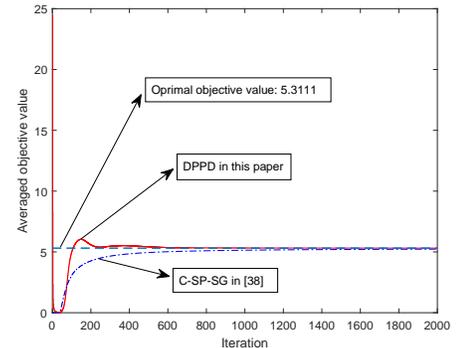}
\caption{Evolutions of $\sum_{l=1}^k L(\bar{x}_{l+1},\bar{\mu}_{l+1})/k$ for DPPD in this paper and $\sum_{i=1}^N L_i(\sum_{l=1}^k x_{i,l}/k,\sum_{l=1}^k\mu_{i,l}/k)$ for C-SP-SG in \cite{mateos2017distributed}, with $Q=1$.}
\label{f3}
\end{figure}

\section*{Acknowledgment}

The authors are grateful to the Editor, the Associate Editor and the anonymous reviewers for their insightful suggestions.

\section*{Appendix}

\subsection{Proof of Lemma \ref{l2}}\label{ap1}

By virtue of the optimality condition for (\ref{ag3-2}) (e.g., see \cite{bertsekas2003convex}), one can obtain that for all $i\in[N]$
\begin{align}
0\in \partial_x L_i(x_{i,k+1},\hat{\mu}_{i,k})+\frac{1}{\alpha_k}(x_{i,k+1}-\hat{x}_{i,k})+\textsc{N}_{X_0}(x_{i,k+1}),      \nonumber
\end{align}
where $\textsc{N}_{X_0}(x_{i,k+1})$ represents the normal cone to the set $X_0$ at the point $x_{i,k+1}$, which amounts to that for all $y\in X_0$
\begin{align}
\Big\langle\partial_x L_i(x_{i,k+1},\hat{\mu}_{i,k})+\frac{1}{\alpha_k}(x_{i,k+1}-\hat{x}_{i,k}),y-x_{i,k+1}\Big\rangle\geq 0,        \label{ss10}
\end{align}
further implying that for all $y\in X_0$
\begin{align}
&\langle x_{i,k+1}-\hat{x}_{i,k},x_{i,k+1}-y\rangle\leq \alpha_k\langle \partial_x L_i(x_{i,k+1},\hat{\mu}_{i,k}),y-x_{i,k+1}\rangle       \nonumber\\
&\hspace{2.8cm}\leq \alpha_k [L_i(y,\hat{\mu}_{i,k})-L_i(x_{i,k+1},\hat{\mu}_{i,k})],            \label{ss11}
\end{align}
where the second inequality follows from the convexity of $L_i$ in the first variable. Note that there exists a positive constant $U_0$ such that $\|\mu_{i,k}\|\leq U_0$ for all $i\in[N]$ and all $k\geq 0$ since $U$ is a bounded set. Subsequently, by setting $y=\hat{x}_{i,k}$ in (\ref{ss11}), one can get that
\begin{align}
\|x_{i,k+1}-\hat{x}_{i,k}\|^2&\leq \alpha_k [L_i(\hat{x}_{i,k},\hat{\mu}_{i,k})-L_i(x_{i,k+1},\hat{\mu}_{i,k})]       \nonumber\\
&\leq \alpha_k\langle \partial_x L_i(\hat{x}_{i,k},\hat{\mu}_{i,k}),\hat{x}_{i,k}-x_{i,k+1}\rangle       \nonumber\\
&\leq \alpha_k S(1+U_0)\|x_{i,k+1}-\hat{x}_{i,k}\|,                                          \nonumber
\end{align}
where we have used the convexity of $L_i$ in the first variable, and the last inequality is due to Cauchy-Schwarz inequality and (\ref{6}).  Therefore, it can be concluded that
\begin{align}
\|x_{i,k+1}-\hat{x}_{i,k}\|\leq S(1+U_0)\alpha_k,                                          \label{ss12}
\end{align}
which leads to
\begin{align}
\|x_{k+1}-(A_k\otimes I_n) x_k\|\leq \sqrt{N}S(1+U_0)\alpha_k,                                          \label{ss13}
\end{align}
where $x_{k}=col(x_{1,k},\ldots,x_{N,k})$ for $k\geq 0$. By defining $\epsilon_k=x_{k+1}-(A_k\otimes I_n) x_k$, it is easy to see that
\begin{align}
x_{k+1}=(A_k\otimes I_n)x_k+\epsilon_k,                 \label{ss14}
\end{align}
where $\|\epsilon_k\|\leq \sqrt{N}S(1+U_0)\alpha_k$ by (\ref{ss13}). Now consider dynamics (\ref{ss14}). Following the same argument as in Lemmas 3 and 4 in \cite{xie2018distributed}, one can obtain that
\begin{align}
\|x_{i,k}-\bar{x}_k\|=O(\alpha_{\lfloor\frac{k}{2}\rfloor}).         \label{ss15}
\end{align}

At the same time, invoking (\ref{5}), (\ref{4s00}) and the nonexpansiveness of projections, one can obtain that
\begin{align}
\|\mu_{i,k+1}-\hat{\mu}_{i,k}\|\leq \alpha_k \|g_i(x_{i,k+1})\|\leq E\alpha_k.         \label{ss0}
\end{align}
As done for obtaining (\ref{ss15}), it can be similarly concluded that
\begin{align}
\|\mu_{i,k}-\bar{\mu}_k\|=O(\alpha_{\lfloor\frac{k}{2}\rfloor}).         \label{ass15}
\end{align}
This completes the proof.

\subsection{Proof of Theorem \ref{t3}}\label{ap2}

Let us focus on the distances from each $x_{i,k}$ and $\mu_{i,k}$ to optimum sets $X^*$ and $U^*$, respectively. Define $p_k:=\sum_{j=1}^N a_{ij,k}P_{X^*}(x_{j,k})$. Then, one has that for each $i\in[N]$
\begin{align}
&\|x_{i,k+1}\|_{X^*}^2\leq\|x_{i,k+1}-p_k\|^2=\|x_{i,k+1}-\hat{x}_{i,k}+\hat{x}_{i,k}-p_k\|^2                    \nonumber\\
&=\|\hat{x}_{i,k}-p_k\|^2+\|x_{i,k+1}-\hat{x}_{i,k}\|^2+2\langle x_{i,k+1}-\hat{x}_{i,k},\hat{x}_{i,k}-p_k\rangle          \nonumber\\
&=\|\hat{x}_{i,k}-p_k\|^2-\|x_{i,k+1}-\hat{x}_{i,k}\|^2             \nonumber\\
&\hspace{0.4cm}+2\langle x_{i,k+1}-\hat{x}_{i,k},x_{i,k+1}-p_k\rangle           \nonumber\\
&\leq \sum_{j=1}^Na_{ij,k}\|x_{j,k}\|_{X^*}^2-\|x_{i,k+1}-\hat{x}_{i,k}\|^2          \nonumber\\
&\hspace{0.4cm}+2\alpha_k [L_i(p_k,\hat{\mu}_{i,k})-L_i(x_{i,k+1},\hat{\mu}_{i,k})],             \label{ss16}
\end{align}
where the last inequality is due to the convexity of norms and (\ref{ss11}) with $y=p_k$.

Similarly, define $q_k:=\sum_{j=1}^N a_{ij,k}P_{U^*}(\mu_{j,k})$, it yields that
\begin{align}
\|\mu_{i,k+1}\|_{U_i^*}^2 &\leq \|\mu_{i,k+1}-q_k\|^2\leq \|\hat{\mu}_{i,k}+\alpha_k g_i(x_{i,k+1})-q_k\|^2           \nonumber\\
&\leq \sum_{j=1}^N a_{ij,k}\|\mu_{j,k}\|_{U^*}^2+\alpha_k^2E^2                 \nonumber\\
&\hspace{0.4cm}+2\alpha_k[L_i(x_{i,k+1},\hat{\mu}_{i,k})-L_i(x_{i,k+1},q_k)],      \label{ss01}
\end{align}
where the second inequality is due to (\ref{4s00}) and the nonexpansiveness of projections, and the convexity of norms and (\ref{5}) have been employed in the last inequality.

Now, define $\phi_{i,k}=\|x_{i,k}\|_{X^*}^2+\|\mu_{i,k}\|_{U^*}^2$ for all $k\geq 0$ and $\phi_k=\sum_{i=1}^N \phi_{i,k}$. Then, summing (\ref{ss16}) and (\ref{ss01}) yields
\begin{align}
\phi_{i,k+1}&\leq \sum_{j=1}^N a_{ij,k}\phi_{j,k}+\alpha_k^2 E^2           \nonumber\\
&\hspace{0.4cm}+2\alpha_k[L_i(p_k,\hat{\mu}_{i,k})-L_i(x_{i,k+1},q_k)].          \label{ss02}
\end{align}

In the following, let us consider the term $L_i(p_k,\hat{\mu}_{i,k})-L_i(x_{i,k+1},q_k)$ in (\ref{ss02}). It is straightforward to see that
\begin{align}
&L_i(p_k,\hat{\mu}_{i,k})-L_i(x_{i,k+1},q_k)=L_i(p_k,\hat{\mu}_{i,k})-L_i(P_{X^*}(\bar{x}_k),\hat{\mu}_{i,k})            \nonumber\\
&\hspace{0.4cm}+L_i(P_{X^*}(\bar{x}_k),\hat{\mu}_{i,k})-L_i(P_{X^*}(\bar{x}_k),\bar{\mu}_k)              \nonumber\\
&\hspace{0.4cm}+L_i(P_{X^*}(\bar{x}_k),\bar{\mu}_k)-L_i(\bar{x}_k,P_{U^*}(\bar{\mu}_{k}))        \nonumber\\
&\hspace{0.4cm}+L_i(\bar{x}_k,P_{U^*}(\bar{\mu}_{k}))-L_i(x_{i,k+1},P_{U^*}(\bar{\mu}_k))        \nonumber\\
&\hspace{0.4cm}+L_i(x_{i,k+1},P_{U^*}(\bar{\mu}_k))-L_i(x_{i,k+1},q_k).            \label{ss17}
\end{align}

To proceed, first observe the following facts:
\begin{align}
\|p_k-P_{X^*}(\bar{x}_k)\|&\leq \sum_{j=1}^N a_{ij,k}\|P_{X^*}(x_{j,k})-P_{X^*}(\bar{x}_k)\|          \nonumber\\
&\leq \sum_{j=1}^N a_{ij,k}\|x_{j,k}-\bar{x}_k\|,          \label{ss03}\\
\|\hat{x}_{i,k}-\bar{x}_k\|&\leq \sum_{j=1}^Na_{ij,k}\|x_{j,k}-\bar{x}_k\|           \nonumber\\
&\leq \sum_{j=1}^Na_{ij,k}\sum_{l=1}^N\frac{1}{N}\|x_{j,k}-x_{l,k}\|      \nonumber\\
&\leq \max_{j,l\in[N]}\|x_{j,k}-x_{l,k}\|.          \label{ss04}
\end{align}

Then, invoking the convexity of $L_i$ in the first argument, (\ref{6}) and $\|\mu_{i,k}\|\leq U_0$ results in
\begin{align}
&|L_i(p_k,\hat{\mu}_{i,k})-L_i(P_{X^*}(\bar{x}_k),\hat{\mu}_{i,k})|\leq S(1+U_0)\|p_k-P_{X^*}(\bar{x}_k)\|                      \nonumber\\
&\hspace{3.2cm}\leq S(1+U_0)\sum_{j=1}^N a_{ij,k}\|x_{j,k}-\bar{x}_k\|,                                     \label{ss18}
\end{align}
where the last inequality is due to (\ref{ss03}). Similarly, one can get that
\begin{align}
&|L_i(\bar{x}_k,P_{U^*}(\bar{\mu}_{k}))-L_i(x_{i,k+1},P_{U^*}(\bar{\mu}_{k}))|          \nonumber\\
&\leq S(1+U_0)\|x_{i,k+1}-\bar{x}_k\|                      \nonumber\\
&\leq S(1+U_0)[\|x_{i,k+1}-\hat{x}_{i,k}\|+\|\hat{x}_{i,k}-\bar{x}_k\|].                      \label{ss19}
\end{align}

Similar to (\ref{ss03}) and (\ref{ss04}), one can obtain that
\begin{align}
\|q_k-P_{U^*}(\bar{\mu}_k)\|&\leq \sum_{j=1}^N a_{ij,k}\|\mu_{j,k}-\bar{\mu}_k\|,               \label{ss003}\\
\|\hat{\mu}_{i,k}-\bar{\mu}_k\|&\leq \max_{j,l\in[N]}\|\mu_{j,k}-\mu_{l,k}\|,                   \label{ss004}
\end{align}
by which invoking similar arguments to (\ref{ss18}) and (\ref{ss19}) implies that
\begin{align}
&|L_i(x_{i,k+1},P_{U^*}(\bar{\mu}_k))-L_i(x_{i,k+1},q_k)|\leq E\|q_k-P_{U^*}(\bar{\mu}_k)\|                      \nonumber\\
&\hspace{0.5cm}\leq E\sum_{j=1}^N a_{ij,k}\|P_{U^*}(\mu_{j,k})-P_{U^*}(\bar{\mu}_k)\|           \nonumber\\
&\hspace{0.5cm}\leq E\sum_{j=1}^N a_{ij,k}\|\mu_{j,k}-\bar{\mu}_k\|,                                     \label{ss018}
\end{align}
\begin{align}
&|L_i(P_{X^*}(\bar{x}_k),\hat{\mu}_{i,k})-L_i(P_{X^*}(\bar{x}_k),\bar{\mu}_{k})|\leq E\|\hat{\mu}_{i,k}-\bar{\mu}_k\|.                      \label{ss019}
\end{align}

By summing (\ref{ss02}) over $i\in[N]$ and making use of (\ref{ss12}), (\ref{ss15})-(\ref{ass15}), (\ref{ss17}), (\ref{ss04})-(\ref{ss19}), (\ref{ss018}) and (\ref{ss019}), one has that
\begin{align}
\phi_{k+1}&\leq \phi_{k}+c_1\alpha_{\lfloor\frac{k}{2}\rfloor}^2+2\alpha_k[L(P_{X^*}(\bar{x}_k),\bar{\mu}_{k})           \nonumber\\
&\hspace{0.4cm}-L(\bar{x}_{k},P_{U^*}(\bar{\mu}_{k}))],                              \label{ss20}
\end{align}
for some constant $c_1>0$. It should be noted that $L(P_{X^*}(\bar{x}_k),\bar{\mu}_{k})-L(\bar{x}_{k},P_{U^*}(\bar{\mu}_{k}))\leq 0$ by the definition of saddle points in (\ref{c3}), and the equality holds if and only if $\bar{x}_k\in X^*$ and $\bar{\mu}_{k}\in U^*$. In addition, because $\alpha_k\to 0$ as $k\to \infty$, for any $\epsilon>0$, there must exist an integer $k_{\epsilon}>0$ such that $\alpha_{\lfloor\frac{k}{2}\rfloor}\leq \epsilon$ for all $k\geq k_\epsilon$. This, together with (\ref{ss20}), yields that for all $k\geq k_{\epsilon}$
\begin{align}
\phi_{k+1}&\leq \phi_{k}+c_1\epsilon\alpha_{\lfloor\frac{k}{2}\rfloor}+2\alpha_k[L(P_{X^*}(\bar{x}_k),\bar{\mu}_{k})           \nonumber\\
&\hspace{0.4cm}-L(\bar{x}_{k},P_{U^*}(\bar{\mu}_{k}))],                              \label{ss21}
\end{align}
which is exactly in the same form as (22) in \cite{liu2017convergence}. As a result, invoking the same reasoning as for (22) in \cite{liu2017convergence} yields that $\lim_{k\to\infty}\phi_k=0$, which together with Lemma \ref{l2} follows that all $x_{i,k}$'s and $\mu_{i,k}$'s will reach a common point in $X^*$ and a common point in $U^*$, respectively, which finishes the proof of the first part.

In the following, it remains to prove the convergence rate of the evaluation error of Algorithm 1. It is easy to observe that the function $L_i(x,\hat{\mu}_{i,k})+\frac{1}{2\alpha_k}\|x-\hat{x}_{i,k}\|^2$ is $1/\alpha_k$-strongly convex in $x$, which leads to that for all $x,y\in X_0$,
\begin{align}
&L_i(x,\hat{\mu}_{i,k})+\frac{1}{2\alpha_k}\|x-\hat{x}_{i,k}\|^2       \nonumber\\
&\geq L_i(y,\hat{\mu}_{i,k})+\frac{1}{2\alpha_k}\|y-\hat{x}_{i,k}\|^2+\frac{1}{2\alpha_k}\|x-y\|^2            \nonumber\\
&\hspace{0.4cm}+\Big\langle\partial_x L_i(y,\hat{\mu}_{i,k})+\frac{1}{\alpha_k}(y-\hat{x}_{i,k}),x-y\Big\rangle.    \label{ss23}
\end{align}
Substituting $x=x^*\in X^*$ and $y=x_{i,k+1}$ into (\ref{ss23}) yields
\begin{align}
&L_i(x^*,\hat{\mu}_{i,k})+\frac{1}{2\alpha_k}\|x^*-\hat{x}_{i,k}\|^2       \nonumber\\
&\geq L_i(x_{i,k+1},\hat{\mu}_{i,k})+\frac{1}{2\alpha_k}\|x_{i,k+1}-\hat{x}_{i,k}\|^2         \nonumber\\
&\hspace{0.4cm}+\frac{1}{2\alpha_k}\|x_{i,k+1}-x^*\|^2,         \label{ss24}
\end{align}
where we have used the fact (\ref{ss10}). It further follows from (\ref{ss24}) that
\begin{align}
&L_i(\bar{x}_{k+1},\bar{\mu}_{k+1})-L_i(x^*,\bar{\mu}_{k})       \nonumber\\
&\leq \frac{1}{2\alpha_k}\sum_{j=1}^N a_{ij,k}\|x_{j,k}-x^*\|^2-\frac{1}{2\alpha_k}\|x_{i,k+1}-x^*\|^2       \nonumber\\
&\hspace{0.4cm} +L_i(x_{i,k+1},\bar{\mu}_{k+1})-L_i(x_{i,k+1},\hat{\mu}_{i,k})            \nonumber\\
&\hspace{0.4cm} +L_i(x^*,\hat{\mu}_{i,k})-L_i(x^*,\bar{\mu}_k)                         \nonumber\\
&\hspace{0.4cm} +L_i(\bar{x}_{k+1},\bar{\mu}_{k+1})-L_i(x_{i,k+1},\bar{\mu}_{k+1}),               \label{ss25}
\end{align}
where the convexity of norms has been used in the inequality. Next, notice that $L_i(x_{i,k+1},\bar{\mu}_{k+1})-L_i(x_{i,k+1},\hat{\mu}_{i,k})=L_i(x_{i,k+1},\bar{\mu}_{k+1})-L_i(x_{i,k+1},\mu_{i,k+1})+L_i(x_{i,k+1},\mu_{i,k+1})-L_i(x_{i,k+1},\hat{\mu}_{i,k})$. By resorting to (\ref{6}) and the convex-concave property of $L_i$ as well as (\ref{ss12}), (\ref{ss15}) and (\ref{ss0})-(\ref{ass15}), it follows, similar to (\ref{ss18})-(\ref{ss19}), that
\begin{align}
L_i(x_{i,k+1},\bar{\mu}_{k+1})-L_i(x_{i,k+1},\hat{\mu}_{i,k})&=O(\alpha_{\lfloor\frac{k}{2}\rfloor}),           \nonumber\\
L_i(x^*,\hat{\mu}_{i,k})-L_i(x^*,\bar{\mu}_{k})&=O(\alpha_{\lfloor\frac{k}{2}\rfloor}).                    \label{ss025}
\end{align}
Meanwhile, similar to (\ref{ss19}), invoking (\ref{ss12}), (\ref{ss15}), and (\ref{ss04}) yields
\begin{align}
L_i(\bar{x}_{k+1},\bar{\mu}_{k+1})-L_i(x_{i,k+1},\bar{\mu}_{k+1})=O(\alpha_{\lfloor\frac{k}{2}\rfloor}).    \label{1v1}
\end{align}
Summing (\ref{ss25}) over $i$ and recalling that $\sum_{i=1}^Na_{ij,k}=1$, in view of (\ref{ss025}) and (\ref{1v1}), one has
\begin{align}
&L(\bar{x}_{k+1},\bar{\mu}_{k+1})-L(x^*,\bar{\mu}_k)       \nonumber\\
&\leq \frac{1}{2\alpha_k}\|x_{k}-\underline{x}^*\|^2-\frac{1}{2\alpha_k}\|x_{k+1}-\underline{x}^*\|^2+c_2\alpha_{\lfloor\frac{k}{2}\rfloor},           \label{ss26}
\end{align}
for some constant $c_2>0$, where $\underline{x}^*:=1_N\otimes x^*$ for brevity. Note that $L(x^*,\bar{\mu}_k)\leq L(x^*,\mu^*)=f^*$, where $\mu^*\in U^*$. Now summing (\ref{ss26}) over finite times $l=1,\ldots,k$ leads to
\begin{align}
&\sum_{l=1}^k L(\bar{x}_{l+1},\bar{\mu}_{l+1})-kf^*       \nonumber\\
&\leq \frac{1}{2\alpha_1}\|x_{1}-\underline{x}^*\|^2+\frac{1}{2}\sum_{l=2}^k\big(\frac{1}{\alpha_l}-\frac{1}{\alpha_{l-1}}\big)\|x_l-\underline{x}^*\|^2              \nonumber\\
&\hspace{0.4cm}-\frac{1}{2\alpha_k}\|x_{k+1}-\underline{x}^*\|^2+c_2\sum_{l=1}^k\alpha_{\lfloor\frac{l}{2}\rfloor},           \label{ss27}
\end{align}
which, together with $\|x_l-\underline{x}^*\|^2\leq 4ND^2$ for $l\in[k]$, implies that
\begin{align}
\frac{\sum_{l=1}^k L(\bar{x}_{l+1},\bar{\mu}_{l+1})}{k}-f^*\leq \frac{2ND^2}{k\alpha_k}+\frac{c_2\sum_{l=1}^k\alpha_{\lfloor\frac{l}{2}\rfloor}}{k}.           \label{ss28}
\end{align}

Similar to (\ref{ss23})-(\ref{ss25}), given that $L_i(x_{i,k+1},\mu)-\frac{1}{2\alpha_k}\|\mu-\hat{\mu}_{i,k}\|^2$ is $1/\alpha_k$-strongly concave in $\mu$, it can be asserted that for each $i\in[N]$
\begin{align}
&L_i(\bar{x}_{k+1},\bar{\mu}_{k+1})-L_i(\bar{x}_{k+1},\mu^*)       \nonumber\\
&\geq -\frac{1}{2\alpha_k}\|\mu^*-\hat{\mu}_{i,k}\|^2+\frac{1}{2\alpha_k}\|\mu^*-\mu_{i,k+1}\|^2       \nonumber\\
&\hspace{0.4cm}+L_i(x_{i,k+1},\mu^*)-L_i(\bar{x}_{k+1},\mu^*)               \nonumber\\
&\hspace{0.4cm}+L_i(x_{i,k+1},\bar{\mu}_{k+1})-L_i(x_{i,k+1},\mu_{i,k+1})   \nonumber\\
&\hspace{0.4cm}+L_i(\bar{x}_{k+1},\bar{\mu}_{k+1})-L_i(x_{i,k+1},\bar{\mu}_{k+1}),             \label{ss29}
\end{align}
where $\mu^*$ is an optimal point in $U^*$. With reference to (\ref{6}), $\|\mu^*\|\leq U_0$ and the convexity of $L_i$ in its first variable, it follows that
\begin{align}
&L_i(x_{i,k+1},\mu_i^*)-L_i(\bar{x}_{k+1},\mu_i^*)\geq -S(1+U_0)\|x_{i,k+1}-\bar{x}_{k+1}\|               \nonumber\\
&\hspace{4.1cm}\geq -c_3\alpha_{\lfloor\frac{k}{2}\rfloor},                                  \label{ss30}
\end{align}
\begin{align}
&L_i(x_{i,k+1},\bar{\mu}_{k+1})-L_i(x_{i,k+1},\mu_{i,k+1})         \nonumber\\
&\hspace{1.5cm}\geq -S(1+U_0)\|\mu_{i,k+1}-\bar{\mu}_{k+1}\|\geq -c_4\alpha_{\lfloor\frac{k}{2}\rfloor},          \label{ss030}
\end{align}
for some constants $c_3,c_4>0$, where (\ref{ss15}) and (\ref{ass15}) are utilized for the last inequalities in (\ref{ss30}) and (\ref{ss030}), respectively. Also, by (\ref{1v1}), it follows that
\begin{align}
L_i(\bar{x}_{k+1},\bar{\mu}_{k+1})-L_i(x_{i,k+1},\bar{\mu}_{k+1})\geq -c_5 \alpha_{\lfloor\frac{k}{2}\rfloor},          \label{2v1}
\end{align}
for some constant $c_5>0$. Then, by substituting (\ref{ss30})-(\ref{2v1}) into (\ref{ss29}) and summing it over $i\in[N]$, one can obtain that
\begin{align}
&L(\bar{x}_{k+1},\bar{\mu}_{k+1})-L(\bar{x}_{k+1},\mu^*)       \nonumber\\
&\geq -\frac{1}{2\alpha_k}\|\underline{\mu}^*-\mu_{k}\|^2+\frac{1}{2\alpha_k}\|\underline{\mu}^*-\mu_{k+1}\|^2-c_0\alpha_{\lfloor\frac{k}{2}\rfloor},             \label{ss31}
\end{align}
where $\underline{\mu}^*:=1_N\otimes \mu^*$ and $c_0:=N(c_3+c_4+c_5)$. Keeping $L(\bar{x}_{k+1},\mu^*)\geq L(x^*,\mu^*)=f^*$ in mind, as done in (\ref{ss27}) and (\ref{ss28}), it can be finally deduced that
\begin{align}
\frac{\sum_{l=1}^k L(\bar{x}_{l+1},\bar{\mu}_{l+1})}{k}-f^*\geq -\frac{2NU_0^2}{k\alpha_k}-\frac{c_0\sum_{l=1}^k\alpha_{\lfloor\frac{l}{2}\rfloor}}{k},             \label{ss32}
\end{align}
which together with (\ref{ss28}) gives rise to
\begin{align}
&\Big|\frac{\sum_{l=1}^k L(\bar{x}_{l+1},\bar{\mu}_{l+1})}{k}-f^*\Big|            \nonumber\\
&\leq \frac{2N\max\{D^2,U_0^2\}}{k\alpha_k}+\frac{\max\{c_0,c_2\}\sum_{l=1}^k\alpha_{\lfloor\frac{l}{2}\rfloor}}{k}.    \label{ss33}
\end{align}
This completes the proof by noting that $\alpha_{\lfloor\frac{l}{2}\rfloor}=O(\alpha_l)$ and $\sum_{l=1}^k\alpha_l=O(\sqrt{k})$ when $\alpha_l=1/\sqrt{l}$.

%%%%%%%%%%%%%%%%%%%%%%%%%%%%%%%%%%%%%%%%%%%%%%%%%%%%%%%%%%%%%%%%%%%%%%%%%%%%%%%%%%%%%%%%%%%%%%%%%%%%%

%%%%%%%%%%%%%%%%%%%%%%%%%%%%%%%%%%%%%%%%%%%%%%%%%%%%%%%%%%%%%%%%%%%%%%%%%%%%%%%%%%%%%%%%%

\end{document}